\documentclass{amsart}

\usepackage[usenames,dvipsnames]{xcolor}
\usepackage{amsmath, amsfonts, amssymb, amsthm, geometry, graphics, graphicx, multicol, multirow, enumerate, float, floatflt, fullpage, tikz}

\newtheorem{thm}{Theorem}

\newcommand{\C}{\mathbb{C}}

\newcommand{\F}{\mathbb{F}}
\newcommand{\al}{\alpha}

\newcommand{\A}{\textnormal{\textbf{A}}}

\newcommand{\Span}[1]{\textnormal{span}\{#1\}}

\title{Cyclic Leibniz Algebras}
\author{Kristin Bugg, Allison Hedges, Minji Lee, Daniel Scofield, and S. McKay Sullivan}
\address{Department of Mathematics, North Carolina State University, Raleigh, NC 27695}
\email{kbugg@ncsu.edu, armcalis@ncsu.edu, mlee9@ncsu.edu, dscofie@ncsu.edu, smsulli4@ncsu.edu}

\begin{document}

\begin{abstract}
Leibniz algebras generated by one element, called cyclic, provide simple and illuminating examples of many basic concepts. It is the purpose of this paper to illustrate this fact.
\end{abstract}

\let\thefootnote\relax\footnote{The work of the authors was supported by NSF grant DMS-0943855.}

\maketitle

\section{Introduction}

Since Leibniz algebras were introduced by Loday \cite{Loday:1993} as a non-commutative generalization of Lie algebras, there has been considerable interest in trying to extend Lie algebra results to this new class of algebras. Leibniz algebras that are generated by one element, called cyclic in \cite{Ray:2012}, are especially illuminating for finding examples. It is the purpose of this work to illustrate this theme. We classify the 3-dimensional cyclic algebras by finding a family and two other cases, one of which is nilpotent. We use this classification to easily find all maximal subalgebras, Frattini subalgebras, Engel subalgebras, and Cartan subalgebras. Engel's theorem holds for Leibniz algebras, as does the characterization of nilpotency in terms of all maximal subalgebras being ideals. We use our 3-dimensional cyclic algebras to show how varying the aforementioned concepts changes the results. We compute the derivation algebras and Killing forms of our examples and show that Cartan's criterion does not hold in one direction. 

\section{Basic Concepts} \label{basics}

We summarize some basic ideas which may be found in \cite{Barnes:2011}, \cite{Demir:2013}, and \cite{Gorbat:2013} along with many other facts. An algebra $\A$ is called Leibniz if it satisfies the identity $x(yz)=(xy)z+y(xz)$. Thus left multiplication is a derivation. If the algebra also satisfies $xy=-yx$, then it is a Lie algebra. Results on low-dimensional classification, nilpotency, solvability, derivations, and semisimplicity of these algebras are found throughout the literature.\\
\indent Cyclic Leibniz algebras, those generated by a single element, are especially basic in structure. In this case $\A$ has a basis $\{a,a^2,\ldots, a^n\}$ and $aa^n=\alpha_1 a+\alpha_2 a^2+\ldots +\alpha_n a^n$. It is easily shown that $\alpha_1=0$. A particularly important ideal of any Leibniz algebra $\A$ is the span of the squares of the elements of $\A$, called $Leib(\A)$. Since $a^j \A=0$ for all $j>1$, $Leib(\A)$ is abelian and it is the smallest ideal of $\A$ for which the quotient is Lie. Thus for cyclic Leibniz algebras the left multiplication by $a$, characterized by the above polynomial, determines the algebra. The matrix for this multiplication is the companion matrix for the polynomial, hence the polynomial is both the characteristic and minimum polynomial for the matrix. Consequences of this are summarized in \cite{Ray:2012}. In particular, the number of maximal subalgebras, Engel subalgebras, and Fitting null components of left multiplications are greatly restricted and easily computed. Thus Frattini subalgebras, Cartan subalgebras, and minimal Engel subalgebras \cite{Barnes:2011} are easily found. Throughout this paper all algebras are taken over the field of complex numbers.

\section{Classification}

Let $\A = \Span{a,a^2,a^3}$ be a Leibniz algebra.  Then $aa^3 = \alpha a^2 + \beta a^3$ for some $\alpha,
\beta \in \mathbb{C}$.  In fact, choosing any $\alpha$
and any $\beta$ in $\mathbb{C}$ yields a Leibniz algebra \cite{Ray:2012}.  However, differing
choices for $\alpha$ and $\beta$ are not guaranteed to yield
non-isomorphic Leibniz algebras.  In this section, we classify the
3-dimensional cyclic Leibniz algebras up to isomorphism.

First, the case when $\alpha=\beta=0$ is the nilpotent cyclic Leibniz
algebra of which there is only one up to isomorphism.  In the case where $\alpha = 0$ and $\beta \neq 0$, let $x =
\frac{1}{\beta}a$.  Then $xx^3 = x^3$.  For the case when $\alpha \neq 0$, let $x =
\frac{1}{\sqrt{\alpha}}a$.  It follows that $xx^3 = x^2 + \frac{\beta}{\sqrt{\alpha}} x^3$. Thus any 3-dimensional non-nilpotent cyclic Leibniz algebra with
$\alpha \neq 0$ has
a generator $x$ which either has multiplication $xx^3 = x^2+\gamma
x^3$ for some $\gamma \in \mathbb{C}$, or multiplication $xx^3 = x^3$.  The following
theorem completes the classification of 3-dimensional cyclic Leibniz
algebras up to isomorphism.

\begin{thm}
\label{classification}
Let $\A$ be a 3-dimensional cyclic Leibniz algebra.  Then $\A$ has
a generator $a$ with one and only one of the following multiplications:
\begin{enumerate}
\item[(i)]{$aa^3 = 0$ ($\A$ is nilpotent),}
\item[(ii)]{$aa^3 = a^3$,}
\item[(iii)]{$aa^3 = a^2 + \gamma a^3$ for precisely one $\gamma \in \mathbb{C}$ such that $0 \leq \arg(\gamma) < \pi.$}
\end{enumerate}
\end{thm}
\begin{proof}
The fact that any 3-dimensional cyclic Leibniz algebra belongs to this
list follows from the arguments given above.  We complete the
proof by showing that each algebra can have no more than one of the
multiplications in this list.

Assume that $\A$ has multiplication $aa^3 = a^3$.  Let
$x = c_1a + c_2 a^2 + c_3 a^3 \in \A$ be a generator.  Then
\begin{align*}
x^3 &= c_1^2(c_1 + c_2 + c_3) a^3 \\
xx^3 &= c_1^3(c_1+ c_2 + c_3) a^3 \\
 &=c_1x^3. 
\end{align*}
Thus $\A$ has no generator with multiplication of type \textit{(iii)}.

Finally, let $\gamma_1 \neq \gamma_2$ be elements of $\mathbb{C}$.   Let $\A$
have multiplication $aa^3 = a^2+\gamma_1 a^3$.  We wish to see if there exists a generator $b = c_1a+c_2a^2+c_3a^3$ such that
\begin{equation}
\label{conditions}
bb^3 = b^2+ \gamma_2 b^3.
\end{equation}
If $\gamma_2 = -\gamma_1$, then $b=-a$ is also a generator for $\A$ satisfying (\ref{conditions}).
In other words generators having multiplication types $aa^3= a^2 + \gamma a^3$ and $bb^3 = b^2 -\gamma b^3$ are contained in the same algebra. Thus the restriction on the argument of $\gamma$ in \textit{(iii)}.

Now assume $\gamma_2 \neq \pm \gamma_1$.  Then (\ref{conditions}) gives a system of equations in $c_1,c_2,c_3$. Every element of $\A^2=\Span{a^2,a^3}$ solves this system trivially.  However, no element of $\A^2$ is a generator. The only solutions not in $\A^2$ are the nonzero multiples of $t=a+\gamma_1 a^2-a^3$. But then $t^2 = 0.$
Thus $t$ is not a generator, and therefore no multiple of $t$ is a generator.  So $\A$ contains no generators which satisfy (\ref{conditions}).  It follows that there is an isomorphism between the
algebra with multiplication $aa^3 = a^2+\gamma_1 a^3$ and the algebra with multiplication $aa^3
= a^2+\gamma_2 a^3$ if and only if $\gamma_2 = \pm \gamma_1$.
\end{proof}

We wish to see how this classification of 3-dimensional cyclic
Leibniz algebras fits into the overall classification of 3-dimensional
Leibniz algebras.  Theorems 6.4 and 6.5 of \cite{Demir:2013} classify the non-lie nilpotent and non-nilpotent Leibniz algebras of dimension 3 respectively.  We investigate which of these
are cyclic.

\begin{thm} Of the five classes listed in \cite[Thm. 6.4]{Demir:2013}, only (1) is a cyclic Leibniz algebra (the nilpotent case).
\end{thm}
\begin{proof} Let $\A = \Span{x,y,z}$ and consider each class, described by its non-zero products.
\begin{enumerate}
  \item $x^2 = y, xy = z$.\\
    \indent Clearly $x, x^2 = y,$ and $x^3 = xx^2 = z$ form a linearly independent basis for $\A$, so $\A$ is cyclic and isomorphic to the nilpotent algebra with $aa^3 = 0$.\\
  \item $x^2 = z$.\\
	\indent Here $\A^2 = \Span{z}$, so $dim(\A^2) = 1$. If $\A$ is cyclic, $dim(\A^2) = 2$. Hence we have a contradiction, and $\A$ is not cyclic.\\
  \item $xy = z, yz = z$.\\
	\indent Similar to (2)\\
  \item $xy = z, yx = -z, y^2 = z$.\\
	\indent Similar to (2).\\
  \item $xy = z, yx = \alpha z, \alpha \in \C \setminus \{1,-1\}$.\\
	\indent Similar to (2).
\end{enumerate}
\end{proof}

\begin{thm} Of the seven classes listed in \cite[Thm 6.5]{Demir:2013}, only (5) (with $\alpha \neq 1$), (6), and (7) are cyclic Leibniz algebras.
\end{thm}
\begin{proof} Let $\A = \Span{x,y,z}$ and consider each class, described by its non-zero products.
\begin{enumerate}
  \item $zx = x$.\\
	\indent Here $\A^2 = \Span{x}$, so $dim(\A^2) = 1$. Thus $\A$ is not cyclic as shown in the previous proof.\\
  \item $zx = \alpha x, \alpha \in \C \setminus \{0\}; zy = y, yz = -y$.\\
	\indent Let $X = \Span{x}$. Then $\A/X$ is a two dimensional ideal with basis $\{y+X,z+X\}$. Also, $(y+X)(y+X) = (z+X)(z+X) = 0+X$ and $(y+X)(z+X) = -(z+X)(y+X)$. Thus $\A/X$ is a Lie algebra. As mentioned above, $Leib(\A)$ is the smallest ideal of $\A$ whose quotient is Lie. So $Leib(\A)$ is no larger than $X$, an ideal of dimension one. But if $\A$ is cyclic, then $Leib(\A) = \A^2$ which has dimension two. Thus we have a contradiction, and $\A$ is not cyclic.\\
  \item $zy = y, yz = -y, z^2 = x$.\\
	\indent Similar to (2).\\
  \item $zx = 2x, y^2 = x, zy = y, yz = -y, z^2 = x$.\\
	\indent Similar to (2).\\
  \item $zy = y, zx = \alpha x, \alpha \in \C \setminus \{0\}$.\\
	\indent Let $t = c_1x+c_2y+c_3z$ be an arbitrary element of $\A$. Then $t^2 = c_1c_3\alpha x+c_2c_3y$ and $t^3 = c_1c_3^2 \alpha^2 x+c_2c_3^2y$. By algebraic manipulation, we can show that $\{t, t^2, t^3\}$ is a linearly independent set when $c_1,c_2,c_3 \neq 0$ and $\alpha \neq 1$. Thus $\A$ is cyclic if and only if $\alpha \neq 1$, in which case it is generated by $t$ for any $t = c_1x+c_2y+c_3z$ with $c_1, c_2, c_3 \neq 0$.\\
  \item $zx = x+y, zy = y$.\\
	\indent Let $t = c_1x+c_2y+c_3z$ be an arbitrary element of $\A$. Then $t^2 = c_1c_3x+c_3(c_1+c_2)y$ and $t^3 = c_1c_3^2x+c_3^2(2c_1+c_2)y$. We find that $\{t, t^2, t^3\}$ is a linearly independent set when $c_1,c_3 \neq 0$. Thus $\A$ is cyclic and generated by $t$ for any $t = c_1x+c_2y+c_3z$ with $c_1, c_3 \neq 0$.\\
  \item $zx = y, zy = y, z^2 = x$.\\
\indent Let $t = c_1x+c_2y+c_3z$ be an arbitrary element of $\A$. Then $t^2 = c_3^2x+c_3(c_1+c_2)y$ and $t^3 = c_3^2(c_1+c_2+c_3)y$. We find that $\{t,t^2,t^3\}$ is a linearly independent set when $c_3 \neq 0$. Thus $\A$ is cyclic and generated by $t$ for any $t = c_1x+c_2y+c_3z$ with $c_3 \neq 0$.
\end{enumerate}
\end{proof}

We have shown that class (5) (with $\alpha \neq 1$), class (6), and
class (7) are cyclic.  For any algebra in class (5), let $a = c_1x+c_2y+c_3z$ be a generator. We compute the multiplication to be $aa^3 = -\alpha c_3^2 a^2 + (\alpha + 1) c_3 a^3$. Then $x = \frac{1}{\sqrt{-\alpha c_3^2}} a = -\frac{1}{c_3 \sqrt{\alpha}} i a$ is also a generator, with multiplication 
$xx^3=x^2-\frac{\alpha+1}{\sqrt{\alpha}}ix^3$.  So the algebras of class (5) are have type \textit{(iii)} multiplication where $\gamma = \pm\frac{\alpha+1}{\sqrt{\alpha}}i$.  Note that $\alpha$
may be chosen to obtain any $\gamma$ except $\gamma = \pm 2i$, since this
corresponds to $\alpha = 1$.\\

For any algebra in class (6) with generator $a = c_1x+c_2y+c_3z$, we calculate the multiplication to be $aa^3 = -c_3^2 a^2 + 2c_3 a^3$.  Taking $x = -\frac{1}{\sqrt{-c_3^2}} a = -\frac{1}{c_3} i a$, we obtain the multiplication $xx^3 = x^2+2ix^3$, which corresponds precisely with the algebra with type \textit{(iii)} multiplication that was missing in class (5).

Finally, for any algebra in class (7) with generator $a$ as above, we calculate the multiplication to be $aa^3 = c_3 a^3$. This algebra also has the generator $x = \frac{1}{c_3} a$ with multiplication
$xx^3=x^3$. So this is the algebra with type $\textit{(ii)}$ multiplication.  Thus all the isomorphism classes from Theorem \ref{classification} have been accounted for in the general classification theorems for 3-dimensional Leibniz algebras found in \cite{Demir:2013}.

\section{Engel and Cartan subalgebras}

Cartan subalgebras for Leibniz algebras were originally defined as nilpotent subalgebras which are their own right normalizer \cite{Omirov:2006}.  However, as was shown in \cite{Barnes:2011}, it is enough that the nilpotent subalgebra be its own normalizer.  A natural
question to ask is whether Cartan subalgebras are always their own left normalizer.  The answer is negative, and
in this section we show that
each 3-dimensional non-nilpotent cyclic Leibniz algebra gives an
example of a Cartan
subalgebra whose left normalizer is not itself.

We define Engel subalgebras as in Lie algebras: for each $x \in \A$ we define
\[
E_\A(x) = \{t \in \A | L_x^k(t) = 0 \text{ for some } k \}.
\]
Then $E_\A(x)$ is a subalgebra \cite{Barnes:2011} called an Engel
subalgebra.  We call an Engel subalgebra a minimal Engel subalgebra if it does not properly contain any other Engel subalgebra.

In \cite{Barnes:2011}, Barnes shows that a subalgebra of a Leibniz algebra is a Cartan subalgebra if and
only if it is minimal Engel.  For cyclic Leibniz algebras this
translates into an unbelievably simple
characterization of Cartan subalgebras, which is given in the following
theorem.  Note that $E_\A(a)$ is also the Fitting null component of $L_a$.  So we will use the notation $A_0 = E_\A(a)$.

\begin{thm}
Let $\A = \Span{a,a^2,\ldots,a^n}$ be a cyclic Leibniz algebra.  Then
$\A_0 = E_\A(a)$ is the unique Cartan subalgebra of $\A$.
\end{thm}

\begin{proof}
Let $x \in A$.  then $x = c_1
a+c_2a^2+\cdots+c_na^n$.  Recall that $L_{a^i}=0$ for all
$i>1$.  Thus $L_x = c_1L_a$.  It follows that 
\[
E_A(x) =\{t \in \A|c_1^kL_{a}^k(t) = 0
\text{ for some } k\}.
\]
Either $c_1 \neq 0$ in which case $E_\A(x) = \A_0$, or
$c_1 = 0$ in which case $L_x=0$ and $E_\A(x) = \A$.  Thus $\A_0$
is the only nontrivial Engel subalgebra of $\A$.  Therefore $\A_0$ is
the only minimal Engel subalgebra.  It follows that $\A_0$
is the unique Cartan subalgebra of $\A$.
\end{proof}

Now we compute $\A_0$
for each of the isomorphism classes given in Theorem
\ref{classification}.  For the nilpotent case, we clearly have $\A_0 = \A$.  
\subsection{Multiplication of Type $aa^3 = a^3$}
Let $\A$ be the 3-dimensional cyclic Leibniz algebra with multiplication $aa^3 = a^3$.  Let $x = c_1 a + c_2
a^2 + c_3 a^3$.  Then 
\begin{align*}
L_a(x) &= c_1 a^2 + (c_2 + c_3) a^3 \\
L_a^2(x) &= (c_1 + c_2 + c_3)a^3.
\end{align*}
Therefore $x_1 = a-a^3$ satisfies $x_1^2 = a^2-a^3 \neq 0$ and $x_1^3= 0$.  Since $x_1$ and $x_1^2$ are linearly independent,  we have
$\A_0 = \Span{x_1,x_1^2}$, the nilpotent cyclic 2-dimensional algebra
generated by $x_1$.  It is interesting to note that $\A_0$ is its own right normalizer by
definition, but given $x = \alpha x_1 + \beta x_1^2 \in \A_0$ we have
\begin{align*}
ax &= a(\alpha x_1 + \beta x_1^2) \\
& = \alpha a(a-a^3)+ \beta a(a^2-a^3)\\
& = \alpha (a^2-a^3) + \beta (a^3 - a^2) \\
& = (\alpha - \beta)(a^2 - a^3) \in \A_0.
\end{align*}
Thus $[\A,\A_0] \in \A_0$.  So $\A$ is the left normalizer of $\A_0$.  Therefore
$\A_0$ is an example of a Cartan subalgebra whose left normalizer is
not itself.  Interestingly, $\A_0$ is also an example of a left ideal
that is not a right ideal.

\subsection{Multiplication of Type $aa^3 = a^2+\gamma a^3$}
Now we consider the algebras with multiplication $aa^3 = a^2 + \gamma
a^3$.  Let $x = c_1 a + c_2 a^2 + c_3 a^3 \in \A_0$.  We have
\begin{align*}
L_a(x) &= (c_1+c_3)a^2 + (c_2 + \gamma c_3) a^3\\
L_a^2(x)&= (c_2 + c_3 \gamma) a^2 + (c_1 + \gamma c_2 + (\gamma^2+1)c_3)a^3.
\end{align*}
Thus $x^3 = 0$ gives the following system of equations
\begin{align*}
c_2 + c_3 \gamma &= 0 \\
c_1 + \gamma c_2 + (\gamma^2+1)c_3 &= 0,
\end{align*}
which simplifies to $c_2 = -\gamma c_3$ and $c_1 = - c_3$.
But with these two conditions, $x$ also satisfies $x^2 = 0$.  We have
\begin{align*}
x &= - c_3 a - c_3 \gamma a^2 + c_3 a^3\\
&= -c_3(a + \gamma a^2 - a^3).
\end{align*}
It follows that $\A_0 = \Span{a+ \gamma a^2 - a^3}$. 

We again have $N^L_A(\A_0) = \A$ since
\[
a(a+\gamma a^2-a^3) = a^2+\gamma a^3 - (a^2 + \gamma a^3) = 0.
\]
We conclude that the left-normalizer of $\A_0$ in any non-nilpotent
cyclic Leibniz algebra of dimension 3 is always $\A$.

\section{Maximal and Frattini Subalgebras} \label{maximal}
Let $\A$ be a cyclic Leibniz algebra generated by $a$, and let $T$ be the matrix for $L_a$ with respect to the basis $\{a, a^2, \ldots, a^n\}$. Then $T$ is the companion matrix for the polynomial $p(x) = x^n - \al_n x^{n-1} - \ldots - \al_2 x = p_1(x)^{n_1} \cdots p_s(x)^{n_s}$, where the $p_j$ are distinct irreducible factors of $p(x)$. Let $r_j(x) = p(x)/p_j(x)$ for $1 \le j \le s$. The maximal subalgebras of $\A$ are the null spaces of each $r_j(T)$, and the Frattini subalgebra $\Phi(\A)$ is the null space of $q(T)$ where $q(x) = p_1(x)^{n_1-1} \cdots p_s(x)^{n_s-1}$ \cite{Ray:2012}.\\
\indent As an example of these results, we directly compute the maximal subalgebras and the Frattini subalgebra for each 3-dimensional algebra from Theorem \ref{classification}. Note that if $\A$ is an $n$-dimensional algebra, it has no more than $n$ maximal subalgebras (since $p(x)$ has no more than $n$ distinct irreducible factors). In this example, the greatest number of maximal subalgebras is three.

\renewcommand{\arraystretch}{1.5}
\begin{table}[htbp]
\caption{Maximal Subalgebras for 3-Dimensional Cyclic Leibniz Algebras}
\centering
\begin{tabular}{c|c|c|c|l}
\cline{2-4}
& {\bf Type} (Thm. \ref{classification}) & {\bf Maximal subalgebras} & ${\bf \Phi(\A)}$ & \\ \cline{2-4}
& \textit{(i)} & $\A^2$ & $\A^2$ &\\ \cline{2-4}
& \textit{(ii)} & $\A^2$, span$\{a-a^2,a-a^3\}$ & span$\{a^2 - a^3\}$ & \\ \cline{1-4}
\multicolumn{1}{|c|}{\multirow{2}{*}{\textit{(iii)}}} & $\gamma = 2i$ & $\A^2$, span$\{ia - a^2, a + a^3\}$ & span$\{ia^2 - a^3\}$ & \\ \cline{2-4}
\multicolumn{1}{|c|}{}  & All other $\gamma$ & $\A^2$, span$\{r_1a^2-a^3,a+r_2a^2\}$, span$\{r_2a^2-a^3,a+r_1a^2\}$ & 0 & \\ \cline{1-4}
\end{tabular}
\end{table}

Recall that if $\A$ is a cyclic Leibniz algebra, then $\A^2 = Leib(\A)$ is abelian. A Leibniz algebra with $\A^2$ nilpotent is elementary if and only if its Frattini subalgebra is trivial \cite{Batten:2013}. Thus the only algebras in the above table which are elementary are those of type \textit{(iii)} with $\gamma \neq 2i$.

\section{Nilpotency}

The lower central series appears immediately in discussions of nilpotent Lie and Leibniz algebras. Implied is the fact that any product of $n$ elements can be written as a linear combination of left normed products which involve the $n$ elements, where a product is left normed if it is of the form  $x_n(x_{n-1}(\ldots (x_3(x_2 x_1))\ldots ))$. Hence if every product of $n$ elements which are left normed is zero, then the algebra is nilpotent. In Leibniz algebras, it is natural to ask if we can replace ``left normed" by ``right normed." In fact, we cannot; this is shown by the non-nilpotent 3-dimensional cyclic Leibniz algebras where any right normed product of three elements is zero. The main theorem about nilpotency for Lie or Leibniz algebras is Engel's theorem, which states that if all left multiplications are nilpotent, then the algebra is nilpotent (\cite{Barnes:2012},\cite{Bosko:2011},\cite{Patsour:2007}). Again, we ask if ``left" can be replaced by ``right." The same non-nilpotent examples cited above show that this is not possible. In summary,

\begin{thm} In Leibniz algebras:
\begin{enumerate}
\item[a.] Not every product can be written as a linear combination of right normed products of the same elements.
\item[b.] If every right normed product of length $n$ is zero, this does not imply that the algebra is nilpotent.
\item[c.] If all right multiplications are nilpotent, this does not imply that the algebra is nilpotent.
\end{enumerate}
\end{thm}

Following famous Lie algebra and group theory results, it has been shown that a Leibniz algebra is nilpotent if and only if all maximal subalgebras are ideals \cite{Demir:2013}. In fact, this holds if all maximal subalgebras are right ideals \cite{Barnes:2011}. A natural question is: can left replace right in this discussion? The answer is negative, as shown the example of the 3-dimensional cyclic algebra $\A$ determined by $aa^3 = a^3$. Recall from Section \ref{maximal} that the maximal subalgebras of $\A$ are $\A^2$ and $\Span{a-a^2, a-a^3}$. Since $L_a(\A^2) = \Span{a^3}$ and $L_a(\Span{a-a^2, a-a^3}) = \Span{a^2-a^3}$, we see that $\A^2$ and $\Span{a-a^2, a-a^3}$ are both left ideals. Thus every maximal subalgebra of $\A$ is a left ideal, but $\A$ is not nilpotent.

\section{Derivations}

Let $\A$ be a Leibniz algebra.  Every left multiplication in $\A$ is a
derivation by definition of the Leibniz product.  We call a derivation
$D$ of $\A$ an inner derivation if
there exists an element $x \in \A$ such that $D(y) = xy$ for all $y
\in \A$.  We denote the space of inner derivations of $\A$ as $\textbf{D}_I(\A)$.
If $\A = \Span{a,a^2,\ldots,a^n}$, then for
each element $x \in \A$ we have $x = c_1a+c_2a^2+\ldots +c_na^n$ for
some $c_1,\ldots,c_n \in \mathbb{C}$.  Thus $L_x = c_1L_a$ where $L_x$
and $L_a$ are the transformations of left multiplication by $x$ and $a$ respectively.  It
follows that the space of inner derivations of $\A$ is precisely
\begin{equation}
\label{derivation}
\textbf{D}_I(\A)=\{cL_a:c \in \mathbb{C} \}.
\end{equation}
We define the set of outer derivations $\textbf{D}_O(\A)$ to
be all derivations of $\A$ which are not inner.

In this section we compute the derivations, both inner and outer, for
all 3-dimensional cyclic Leibniz algebras.  In each of the following
cases, let $\A$ be the 3-dimensional cyclic Leibniz algebra with the
given multiplication, and let $D$ be a general derivation characterized by $D(a) =
\alpha_1 a+ \alpha_2 a^2 + \alpha_3 a^3$.

\subsection{Nilpotent Case: $aa^3 = 0$}
We calculate 
\begin{align*}
D(a^2)&=aD(a)+D(a)a=2 \alpha_{1}a^2 + \alpha_{2}a^3,\\
D(a^3)&=aD(a^2)+D(a)a^2 =3 \alpha_{1}a^3,\\
D(aa^3)&=aD(a^3)+D(a)a^3 = 0.
\end{align*}
Thus there are no constraints on $\alpha_1,$ $\alpha_2,$ and
$\alpha_3$.  The space of derivations of $\A$ is given by
\[
\textbf{D}(\A) = \left\{
\left(
\begin{array}{ccc}
\alpha_1 & 0 & 0 \\
\alpha_2 & 2\alpha_1 & 0 \\
\alpha_3 & \alpha_2 & 3 \alpha_1 \\
\end{array}
\right): \alpha_1,\alpha_2,\alpha_3 \in \mathbb{C}
 \right\}.
\]
So by (\ref{derivation}) we know that $\textbf{D}_I(\A)$ is the subset of $\textbf{D}$
with $\alpha_1 = \alpha_3 = 0$.  It follows that $\textbf{D}_O(\A)$ is the
subset of $\textbf{D}(\A)$ with $\alpha_1 \neq 0$ or $\alpha_3 \neq 0$.

\subsection{Multiplication type $aa^3 = a^3$}

We have 
\begin{align*}
D(a^2)&=2 \alpha_{1}a^2 +( \alpha_{2}+ \alpha_{3})a^3\\
D(a^3)&=( 3 \alpha_{1} + \alpha_{2}+ \alpha_{3})a^3\\
D(aa^3)&= (4 \alpha_{1} + \alpha_{2}+ \alpha_{3})a^3
\end{align*}
So from the relation $D(aa^3) = D(a^3)$ we obtain the restriction
$\alpha_{1}=0$.  Thus the space of derivations of $\A$ is given by
\[
\textbf{D}(\A) = \left\{
\left(
\begin{array}{ccc}
0 & 0 & 0 \\
\alpha_2 & 0 & 0 \\
\alpha_3 & \alpha_2+\alpha_3 & \alpha_2+\alpha_3 \\
\end{array}
\right): \alpha_2,\alpha_3 \in \mathbb{C}
 \right\}.
\]
So by (\ref{derivation}) we know that $\textbf{D}_I(\A)$ is the subset of $\textbf{D}(\A)$
with $\alpha_3 = 0$, and $\textbf{D}_O(\A)$ is the subset with $\alpha_3
\neq 0$.

\subsection{Multiplication type $aa^3 = a^2 + \gamma a^3$}
We compute
\begin{align*}
D(a^2)&=(2 \alpha_{1} + \alpha_{3})a^2 + (\alpha_{2}+\gamma
\alpha_{3})a^3\\
D(a^3)&=(\alpha_{2} + \gamma \alpha_{3})a^2 + [3 \alpha_{1} + \gamma
\alpha_{2} + (\gamma^{2} +1) \alpha_{3}]a^3\\
D(aa^3)&= [4 \alpha_1+\gamma \alpha_2+(\gamma^2+1)\alpha_3]a^2
+[4 \gamma \alpha_1+(\gamma^2+1)\alpha_2+(\gamma^3+2\gamma)\alpha_3]a^3
\end{align*}
The relation $D(aa^3) =D(a^2+\gamma a^3)$ yields the constraint
$\alpha_1 = 0$.  Therefore the space of derivations of $\A$ is
given by
\[
\textbf{D}(\A) = \left\{
\left(
\begin{array}{ccc}
0 & 0 & 0 \\
\alpha_2 & \alpha_3 & \alpha_2+\gamma\alpha_3 \\
\alpha_3 & \alpha_2+\gamma\alpha_3 & \gamma\alpha_2+(\gamma^2+1)\alpha_3 \\
\end{array}
\right): \alpha_2,\alpha_3 \in \mathbb{C}
 \right\}.
\]
It follows by (\ref{derivation}) that $\textbf{D}_I(\A)$ is the subset of
$\textbf{D}(\A)$ with $\alpha_3 = 0$, and $\textbf{D}_O(\A)$ is the subset with $\alpha_3 \neq 0$.

\section{Semisimplicity}

 In this section we recall the definition of semisimplicity for Leibniz algebras and also the Killing form \cite{Demir:2013}. We compute the Killing form in each of our examples and compute the radicals of each form. It is known that if $\A$ is semisimple then the radical of the Killing form, $\A^{\perp}$, is $Leib(\A)$. The converse result may not hold, as was shown by a cyclic example in \cite{Demir:2013}. All cyclic algebras have $Leib(\A)=\A^{2}$ and $\A^{2}$ abelian. Notice from this discussion that when $\A$ is semisimple, $\A^{\perp}=rad(\A)=Leib(\A)$. We will use our examples to show that neither $rad(\A)=\A^{\perp}$ nor $Leib(\A)=\A^{\perp}$ is sufficient to give that $\A$ is semisimple.

A Lie algebra $L$ is said to be semisimple if $rad(L)=\{0\}$. For a finite dimensional Lie algebra $L$ over a field $\F$, the Killing form $\kappa( , ) : L\times L \longrightarrow \F$ is defined by $\kappa(a, b)=tr(L_{a}L_{b})$ for all $a, b \in L$ and the radical of the Killing form $\kappa( , )$ on $L$ is defined by $L^{\perp}$=$\{b\in L |  \kappa(b,a)=0$ for all $a \in L\}$. The Killing form and the radical of the Killing form for a Leibniz algebra are defined as for Lie algebras. 

\subsection{Multiplication of Type $aa^3 = 0$}

It is clear that $\A=\Span{a, a^2, a^3}$ is a nilpotent Leibniz algebra. Then in some basis the left multiplication operators $L_{x}$ are upper-triangular nilpotent matrices by Engel's theorem \cite{Gorbat:2013}. So, the Killing form $\kappa( , )$  is trivial.

\subsection{Multiplication of Type $aa^3 = a^3$}

 Let $x=\alpha a +b$, $y=\beta a +c$, $x, y \in \A$ and $b, c \in \A^{2}$. Then
\[L_{x}=\alpha
\left(
\begin{array}{ccc}
0 & 0 & 0 \\
1 & 0 & 0 \\
0 & 1 & 1\\
\end{array}
\right), 
 L_{y}=\beta
\left(
\begin{array}{ccc}
0 & 0 & 0 \\
1 & 0 & 0 \\
0 & 1 & 1\\
\end{array}
\right),
 L_{x}L_{y}=\alpha \beta
\left(
\begin{array}{ccc}
0 & 0 & 0 \\
0 & 0 & 0 \\
1 & 1 & 1\\
\end{array}
\right), 
 \kappa(y, x) = tr(L_{y}L_{x})=\alpha \beta
 \]
If $\alpha \beta = 0$, then the Killing form on $\A$ is trivial.  If $\alpha\beta \neq 0$, $\A^{\perp} = \Span{\alpha_{2}a^{2}+\alpha_{3}a^{3}}=\A^{2}=Leib(\A)$. It is clear that $\A$ is solvable with $rad(\A) = \A$; hence $\A$ is not semisimple. 

\subsection{Multiplication of Type $aa^3 = a^2 + \gamma a^3$}

 Let $x=\alpha a +b$, $y=\beta a +c$, $x, y \in \A$ and $b, c \in \A^{2}$. Then\\
\[
L_{x}=\alpha
\left(
\begin{array}{ccc}
0 & 0 & 0 \\
1 & 0 & 1 \\
0 & 1 & \gamma\\
\end{array}
\right), 
 L_{y}=\beta
\left(
\begin{array}{ccc}
0 & 0 & 0 \\
1 & 0 & 1 \\
0 & 1 & \gamma\\
\end{array}
\right),
 L_{x}L_{y}=\alpha \beta
\left(
\begin{array}{ccc}
0 & 0 & 0 \\
0 & 1 & \gamma \\
1 & \gamma & \gamma^{2} +1 \\
\end{array}
\right), 
 \kappa(y, x) = tr(L_{y}L_{x})=\alpha \beta(\gamma^{2} + 2)
 \]
The Killing form on $\A$ is trivial for $\alpha \beta = 0$.  If $\alpha \beta \neq 0$ and $\gamma \neq \pm \sqrt{2}i$, then $\kappa(y,x)=tr(L_{y}L_{x}) \neq 0$. Thus, $\A^{\perp}=\A^2 =Leib(\A)$. This shows $Leib(\A)=\A^{\perp}$ does not imply that $\A$ is semisimple.  On the other hand, if $\alpha \beta \neq 0$ and $\gamma = \pm \sqrt{2}i$, then $\kappa(y,x)=tr(L_{y}L_{x})=0$. It follows that $\A^{\perp}=\A$ and the Killing form on $\A$ is trivial. This shows that $rad(\A)=\A^{\perp}$ does not imply that $\A$ is semisimple.

\section*{Acknowledgements}
The authors are graduate students at North Carolina State University. Their work was funded by an REG grant from the National Science Foundation. The authors would like to thank Dr. Ernest Stitzinger for his guidance and support.

\bibliography{main}{}
\bibliographystyle{plain}

\end{document}